\documentclass[a4paper,10pt]{amsart}
\usepackage[foot]{amsaddr}
\usepackage{geometry}
\usepackage{amsfonts}
\usepackage{amsthm}
\usepackage{amssymb}
\usepackage{amsmath}
\usepackage{mathrsfs}
\usepackage{comment}
\usepackage{theoremref}
\usepackage{enumerate}
\usepackage{bbm}
\usepackage{bm}
\usepackage{comment}
\usepackage{mathtools}
\usepackage{a4wide}
\usepackage{xcolor}
\usepackage{lipsum}

\makeatletter
\newcommand{\proofpart}[2]{%
    \par
  \addvspace{\medskipamount}%
  \noindent\emph{Step #1: #2}\par\nobreak
  \addvspace{\smallskipamount}%
  \@afterheading
}
\makeatother

\DeclarePairedDelimiter\abs{\lvert}{\rvert}%
\DeclarePairedDelimiter\norm{\lVert}{\rVert}%

\makeatletter
\let\oldabs\abs
\def\abs{\@ifstar{\oldabs}{\oldabs*}}
\let\oldnorm\norm
\def\norm{\@ifstar{\oldnorm}{\oldnorm*}}
\makeatother

\makeatletter
\g@addto@macro\bfseries{\boldmath}
\makeatother

\newcommand{\T}{\mathbb{T}}

\newcommand{\Ka}{\mathcal{K}}
\newcommand{\conj}[1]{\overline{#1}}
\newcommand{\D}{\mathbb{D}}

\newcommand{\dist}[2]{\text{dist}( #1, #2 ) }

\renewcommand\Re{\operatorname{Re}}

\newcommand{\supp}[1]{\text{supp}({#1})}

\newtheorem{thm}{Theorem}[section]
\newtheorem{lemma}[thm]{Lemma}

\newtheorem{cor}[thm]{Corollary}
\newtheorem{prop}[thm]{Proposition}

\theoremstyle{definition}

\theoremstyle{definition}

\newcommand{\Addresses}{{% additional braces for segregating \footnotesize
		\bigskip
		\footnotesize
		
		Adem Limani, \\ \textsc{Centre for Mathematical Sciences \\ Lund University \\
		Lund, Sweden}\\
		\texttt{adem.limani@math.lu.se}
		
	}}

\begin{document}
\title{\textbf{A generic threshold phenomena in weighted $\ell^2$}} %for Cauchy integrals in $\ell^p$}}

\author{Adem Limani} 
\address{Centre for Mathematical Sciences, Lund University, Sweden}
\email{adem.limani@math.lu.se}

\date{\today}

\begin{abstract}
We consider threshold phenomenons in the context of weighted $\ell^2$-spaces. Our main result is a summable Baire category version of K\"orner's topological Ivashev-Musatov Theorem, which is proved to be optimal from several aspects.

%We consider a uniqueness problem on Fourier coefficients of normalized Cauchy transforms. These problems intrinsically entail in proving a simultaneous approximation phenomenon and the existence of cyclic inner functions, in certain sequence spaces. Our results have several applications in different direction. First, we provide a new non-probabilistic proof of a classic theorem by Kahane and Katzenelson on simultaneous approximation. Secondly, we illustrate the lack of uniform admissible majorants de Branges-Rovnyak spaces for Fourier coefficients of for Beurling-Malliavin type problems in the context of de Branges-Rovnyak spaces.

%We consider a uniqueness problem for Fourier coefficients of normalized Cauchy transforms. This problem entails in proving a simultaneous approximation phenomenon and the existence of cyclic inner functions, both in certain sequence spaces. Our developments allow us new non-probabilistic proof of a classic theorem by Kahane and Katzenelson on simultaneous approximation. Furthermore, we illustrate consequences of our results to Beurling-Malliavin type problems in the context of de Branges-Rovnyak spaces.

\end{abstract}

\maketitle
\section{Introduction}\label{SEC:INTRO}
\subsection{Threshold phenomenons in weighted $\ell^2$}
Given a sequence of positive numbers $w=(w_n)_{n=0}^\infty$, we denote by $\ell^2(w)$ the Hilbert space of distributions $S$ on $\T$ satisfying 
\[
\norm{S}_{\ell^2(w)} = \left( \sum_{n} \abs{\widehat{S}(n)}^2 w_{|n|} \right)^{1/2}< \infty.
\]
It is straightforward to verify that $\ell^2(w)$ is separable, containing the trigonometric polynomials as a dense subset. If $w_n=1$ is constant, then we retain the the classical $\ell^2$-space, which by Parseval's Theorem can be identified with $L^2(\T,dm)$, the space of square integrable functions $f$ on the unit-circle $\T$:
\[
\sum_n \abs{\widehat{f}(n)}^2 = \int_{\T} \abs{f(\zeta)}^2 dm(\zeta) < \infty,
\]
where $dm$ denotes the arc-length measure on $\T$. Now it is of interest to exhibit function theoretical properties of elements belonging to $\ell^2(w)$. %For instance, when $w_n = (1+|n|)^{\alpha}$ with $\alpha>0$, we encounter the fractional Sobolev spaces $\ell^2_\alpha$, which play crucial role in modern analysis. 
A natural question that arises is: \emph{when does $\ell^2(w)$ contain the continuous functions on $\T$?} The following observation clarifies that Parseval's theorem gives rise to the following threshold phenomena.

\begin{thm}\thlabel{THM:l2wcont} For $(\lambda_n)_n$ be positive numbers with $\lambda_n \uparrow +\infty$. Then there exists a continuous function $f$ on $\T$, such that
\[
\sum_n \, \abs{\widehat{f}(n)}^2 \lambda_{|n|} = \infty.
\]
\end{thm}
This is observation is certainly expected, and likely well-known to experts. For the readers convenience, we outline a short proof in the Appendix. Likewise, we may ask: \emph{when does $\ell^2(w)$ consist only of nicely behaved elements?} Again, in view of Parseval's Theorem, if $w_n \downarrow 0$, one may expect the corresponding space $\ell^2(w)$ to contain singular measures wrt $dm$. In fact, the following result asserts something stronger.

\begin{thm}\thlabel{COR:l2wlac} Let $(w_n)_n$ positive numbers with $w_n \downarrow 0$. Then there exists a sequence of positive numbers $(a_n)_n$ satisfying 
\[
\sum_n a^2_n w_{|n|} < \infty,
\]
such that $(a_n)_n$ are not the Fourier coefficients of a complex finite Borel measure on $\T$.
\end{thm}
It is well-known that \thref{COR:l2wlac} is essentially a dual reformulation of \thref{THM:l2wcont}. Our main purpose in this note, is to find the precise threshold for when $\ell^2(w)$ always consists of continuous functions in $\T$, thus complementing the observations phrased above.

%In both these cases, we see that Parseval's theorem gives rise to a threshold phenomenon. 

%Now our main intention is to reverse the perspective and find the sharp conditions on $w$, for which $\ell^2(w)$ only consists of continuous functions and/or measures on $\T$.

%Prove these three results (only the last one remains and submit to Comptes Rendus) (make less than 10 pages) Keep everything short, including the intro.

%In the regime $\lambda_n \uparrow \infty$, we still seem to have room for discontinuous functions in the corresponding space $\ell^2(\lambda)$. What is the precise threshold between continuity and Dirichlet spaces? For instance, the classical Dirichlet space 
%\[
%\sum_n \abs{\widehat{f}(n)}^2 (1+|n|) < \infty,
%\]
%contains discontinuous functions (think conformal mappings with slits accumulating, but finite area!)

\subsection{Topologically bad support}
Our main result is a weighted $\ell^2$-analogue of K\"orner's topological Ivashev-Musatov Theorem in \cite{korner2003topological}.

\begin{thm}\thlabel{THM:BADSUPP} Let $(\lambda_n)_n$ be positive numbers satisfying the following hypothesis:
\begin{enumerate}
    \item[(i)] $\sum_n \frac{1}{\lambda_n} = +\infty$,
    \item[(ii)] there exists $C>1$, such that for all $n\geq 1$:
    \[
    C^{-1} \lambda_n \leq \lambda_k \leq C \lambda_n, \qquad n\leq k \leq 2n.
    \]
\end{enumerate}
%\begin{equation}\label{EQ:DIVCOND}
%\sum_n \frac{1}{\lambda_n} = +\infty.
%\end{equation}
Then there exists a positive function $f\in L^\infty(\T,dm)$ such that $\supp{fdm}$ has empty interior, but 
\[
\sum_n \abs{\widehat{f}(n)}^2 \lambda_{|n|} < \infty.
\]
\end{thm}
Our theorem is essentially sharp, since if the hypothesis in $(i)$ is violated, then Cauchy-Schwartz inequality implies
\[
\sum_n \abs{\widehat{f}(n)} \leq \left( \sum_n \abs{\widehat{f}(n)}^2 \lambda_n \right)^{1/2}\left( \sum_n \frac{1}{\lambda_n} \right)^{1/2} < \infty,
\]
hence any element in $\ell^2(\lambda)$ belongs to the so-called Wiener algebra, and are thus continuous on $\T$. Primary examples of space for which our \thref{THM:BADSUPP} applies to are Dirichlet-type space of elements $f\in L^2(\T, dm)$ with
\[
\sum_n \abs{\widehat{f}(n)}^2 (1+|n|)^{\gamma} < \infty, \qquad 0<\gamma \leq 1.
\]
In this classical framework, it is a well-known fact that one can exhibit elements which are not continuous, while if $\gamma>1$, classical Sobolev embeddings ensure containment into H\"older-type spaces. For instance, see \cite{evans2022partial}. At the end of Section 2, we shall explain why \thref{THM:BADSUPP} cannot simply be derived from methods involving sparse Fourier support, such as Sidon sets. For instance, this includes Riesz-type products and lacunary series.

A natural follow-up question is whether the regularity hypothesis $(ii)$ is truly necessary, or if it can be replaced by a more natural monotonicity condition. Our next observation clarifies that this is not the case, as condition $(ii)$ cannot simply be omitted.

\begin{thm}\thlabel{THM:(ii)} There exists a sequence of positive, increasing numbers $(\lambda_n)_n$ with $\lambda_n \uparrow +\infty$ and $\sum_n \frac{1}{\lambda_n} = + \infty$, such that whenever $S$ is a distribution in $\T$ with 
\[
\sum_n \abs{\widehat{S}(n)}^2 \lambda_{|n|}< \infty,
\]
then $\supp{S}= \T$.
\end{thm}
This results is essentially similar to to K\"orner's Theorem 1.2 in \cite{korner1977theorem2}, and its proof will be inspired from it. It remains unclear what a "critical" version of \thref{THM:BADSUPP} would entail, which alligns well with other uncertainty principles in harmonic analysis, such as the Ivashev-Musatov Theorem and the Beurling-Malliavin Multiplier Theorem. For further discussions on related topics, see \cite{korner1986theorem3} and \cite{havinbook}.

 %However, we shall in Section \ref{SEC:3} clarify how techniques involving sparse Fourier support can be utilized to give simpler results of similar flavor to \thref{THM:BADSUPP}. Details about discussions and statements are deferred to therein.

%We remark that elementary considerations show that such functions cannot have lacunary sparse Fourier supports, hence lacunary Fourier series are not employable.
\subsection{A Baire category version}
It turns out that the elements appearing in the statement of \thref{THM:BADSUPP} are generic. To this end, let $\mathscr{C}$ denote the collection of non-empty compact subsets of $\T$ equipped with the so-called \emph{Hausdorff metric} 
\[
d_{\mathscr{C}}(E,K) = \sup_{\zeta \in E} \dist{\zeta}{K} + \sup_{\xi\in K} \dist{\xi}{E}.
\]
It is straightforward to verify that $(\mathscr{C}, d_{\mathscr{C}})$ is a complete metric space. We can actually offer the following Baire category version of \thref{THM:BADSUPP}.

\begin{thm}\thlabel{THM:BAIRETOP} Let $(\lambda_n)_n$ be positive numbers satisfying the hypothesis $(i)-(ii)$ of \thref{THM:BADSUPP}. Consider the collection $\mathcal{L}_{\mathscr{C}}(\lambda) \subset L^\infty(\T,dm) \times \mathscr{C}$ of ordered pairs $(f,E)$ satisfying the properties:
\begin{enumerate}
    \item[(i)] $\supp{\mu}\subseteq E$,
    \item[(ii)] $\{\widehat{f}(n) \}_n \in \ell^2(\lambda)$,
\end{enumerate}
equipped with the metric 
\[
d_\lambda \left( (f,E), (g,K) \right) := d_{\mathscr{C}}(E,K) + \norm{f-g}_{\ell^2(\lambda)}.
\]
Then the sub-collection of pairs $(f,E) \in \mathcal{L}_{\mathscr{C}}(\lambda)$ with $E$ having no interior is generic in the complete metric space $\left(\mathcal{L}_{\mathscr{C}}(\lambda), d_\lambda \right)$.
\end{thm}
Our principal emphasis will be to prove \thref{THM:BAIRETOP}, from which \thref{THM:BADSUPP} follows as an immediate corollary.

\subsection{Organization and notation}
The paper is organized as follows: Section \ref{SEC:2} is devoted to proving \thref{THM:BAIRETOP}, where we take inspiration from work of T. W. K\"orner in \cite{korner2003topological}. In Section \ref{SEC:3}, we deduce \thref{THM:(ii)} from a slightly stronger result, which utilizes a simple but useful lemma of K\"orner and Meyer. In the Appendix, we give short sketches of proofs for \thref{THM:l2wcont} and \thref{COR:l2wlac}.
%In Section \ref{SEC:3}, we demonstrate how methods involving sparse Fourier support can be used to obtain weaker, yet conceptually related results to \thref{THM:BADSUPP}. In the Appendix, we sketch the short proofs of \thref{THM:l2wcont} and \thref{COR:l2wlac}.

For two positive numbers $A, B >0$, we will frequently use the notation $A \lesssim B$ to mean that $A \leq cB$ for some positive constant $c>0$. If both $A\lesssim B$ and $B \lesssim A$ hold, we will write $A\asymp B$. %Additionally, absolute constants will generally be denoted by $C$, even though the value of $C$ may vary from line to line.

\subsection{Acknowledgments} This research was supported by a stipend from the 
Knut \& Alice Wallenberg Foundation (grant no. 2021.0294). The author is grateful to Alexandru Aleman and Eskil Rydhe for supportive discussions.

%(Prove a lack of Piatetski-Shapiro phenomenon in $\ell^2(w)$, so closer to potential theory due to the Hilbert space structure! Perhaps one-vs two-sided problem? Save for separate paper!).

\section{Bounded functions with bad support}\label{SEC:2}

\subsection{The doubling condition}
The following lemma summarizes how the hypothesis $(ii)$ of \thref{THM:BADSUPP} will be used throughout. By means of substituting $\lambda_n$ with $\max(\lambda_n, 1)$, we may without loss of generality always assume that $\lambda_n \geq 1$ for all $n$.

\begin{lemma} \thlabel{LEM:DOUBLE}
Let $(\lambda_n)_n$ be positive numbers with the hypothesis $(ii)$ of \thref{THM:BADSUPP}. Then there  exists $M(\lambda)>1$ such that for any integer $M> M(\lambda)$, the following statements hold:
\begin{enumerate}
    \item[(a.)] \[
(1+n)^{-M}\lambda_n \leq (1+m)^{-M}\lambda_m \qquad 1\leq m<n.
\]
\item[(b.)]
\[
\sum_{j=1}^n \frac{j^{M-1}}{\lambda_j} \leq 10M(\lambda) \frac{n^M}{\lambda_n}, \qquad n=1,2,3,\dots
\]
\item[(c.)]
\[
\sum_{j>n} \frac{1}{\lambda_j j^{M+1}} \leq M(\lambda) \frac{1}{n^M  \lambda_n} , \qquad n=1,2,3,\dots
\]
\end{enumerate}

\end{lemma}
Note the the second condition means that $\lambda_n (1+n)^{-M}$ is non-increasing whenever $M>M(\lambda)$, which we shall use frequently. In particular, this implies that $\lambda_n$ has at most polynomial growth.

\begin{proof}
The proof is simple. Let $C(\lambda)>1$ be the constant appearing in the hypothesis $(ii)$ of \thref{THM:BADSUPP}. Repeatedly using the assumption $(ii)$ on $(\lambda_n)_n$, and Peetre's inequality, we get
\begin{multline*}
\lambda_{n} (1+n)^{-M} \leq \lambda_m C(\lambda)^{\log_2(n-m)} (1+n)^{-M} \leq \lambda_m  (1+n-m)^{M(\lambda)} (1+n)^{-M}  \\ \leq \lambda_m  (1+m)^{-M} \left(\frac{1+m}{1+n}\right)^{M-M(\lambda)} \leq (1+m)^{-M} \lambda_m, \qquad 1\leq m < n,
\end{multline*}
whenever $M>M(\lambda)$. The claim in $(b.)$ readily follows from $(a.)$: $j^{M-1}/\lambda_j$ being non-decreasing. In order to prove $(c.)$, we first note that 
\[
\lambda_{2^j} 2^{jA} \leq C(\lambda)2^{-A} \lambda_{2^{j+1}} 2^{j+1} \leq \lambda_{2^{j+1}} 2^{j+1}, \qquad j=0,1,2,\dots 
\]
whenever $A> 10 \log C(\lambda)$. With this at hand, we get that for any $n\geq 1$:
\[
\sum_{j>n} \frac{1}{\lambda_j j^{M+1}} = \sum_{j> \log n} \sum_{k\asymp 2^j} \frac{1}{\lambda_k k^{M+1}} \asymp \sum_{j> \log n} \frac{1}{\lambda_{2^j} 2^{jM}} \lesssim \frac{1}{\lambda_n n^{A}} \sum_{j>\log n} 2^{-j(M-A)} \lesssim \frac{1}{\lambda_n n^{M}},
\]
whenever $M>A$.
\end{proof}

\subsection{A smooth localizing function}

The proof of our result hinges on the following principal lemma on uniformly bounded functions with small amplitudes.

\begin{lemma}\thlabel{LEM:KEYKÖR} Let $(\lambda_n)_n$ positive numbers satisfying the hypothesis $(i)-(ii)$ of \thref{THM:BADSUPP}. Then for there exists an integer $N(\lambda)>0$ such that for any $N\geq N(\lambda)$, the following statement holds: for any $0<\varepsilon<1$, there exists $\psi_{\varepsilon} \in C^\infty(\T)$, such that 
\begin{enumerate}
    \item[(i)] $0\leq \psi_\varepsilon\leq 1+\varepsilon$ on $\T$,
    \item[(ii)] $\psi_{\varepsilon}=0$ in a neighborhood of $1$, whose length tends to zero as $\varepsilon \to 0$,
    \item[(iii)] $\int_{\T} \psi_\varepsilon dm = 1$,
    \item[(iv)] $\sup_{n\neq 0} \abs{\widehat{\psi_\varepsilon} (n)} \leq \min \left( \varepsilon, c(N) \varepsilon^{1-N} \abs{n}^{-N} \right)$,
    \item[(v)]
    $\sum_{n\neq 0} \abs{\widehat{\psi}_\varepsilon (n)}^2 \lambda_{|n|} \leq \varepsilon^2$.
\end{enumerate}
    
\end{lemma}

In order to prove our lemma we shall need the following building-block from the work of T. W. K\"orner, see Lemma 20 in \cite{korner2003topological}.

\begin{lemma}[T. W. K\"orner] \thlabel{LEM:KORNER} Given positive integers $M, S>0$, there exists constants $A(M,S)>0$ and $\delta(M,S)>0$ such that the following statement holds: for any $0<\eta<1/2$, one can find smooth real-valued functions $g_{\eta,M,S}$ satisfying the conditions
\begin{enumerate}
    \item[(i)] $g_{\eta,M,S}(\zeta)=1$, for $|\zeta-1|\leq \delta(M,S)\eta$,
    \item[(ii)] $-S^{-1}\leq g_{\eta,M,S} \leq 1 $ on $\T$,
    \item[(iii)] $\widehat{g}_{\eta,M,S}(0)=0$,
    \item[(iv)] 
    \[
    \abs{\widehat{g}_{\eta,M,S}(n)} \leq \eta A(M,S) \min \left( (\eta|n|)^M , (\eta|n|)^{-M} \right), \qquad n\neq 0
    \]
\end{enumerate}
\end{lemma}
Furthermore, the parameters $A(M,S), \delta(M,S) \to 0$ as $S \to \infty$.

\begin{proof}[Proof of \thref{LEM:KEYKÖR}]
Fix a number $M> M(\lambda)$ as in the statement of \thref{LEM:DOUBLE}. Let $0<\varepsilon<1$, pick an integer $S=S(\varepsilon) > 2/\varepsilon$, such that 
\begin{equation}\label{EQ:Leps}
L(\varepsilon) := \sum_{1/\varepsilon \leq j \leq S(\varepsilon)} \frac{1}{\lambda_j} \geq 1/\varepsilon. 
\end{equation}
This is possible due to hypothesis $(i)$ of \thref{THM:BADSUPP}. Now applying \thref{LEM:KORNER}, we consider functions of the form 
\[
\psi_\varepsilon(\zeta) = 1- \frac{1}{L(\varepsilon)} \sum_{1/\varepsilon \leq j \leq S(\varepsilon)} \frac{1}{\lambda_j} g_{1/j, S(\varepsilon), M}(\zeta),
\]
where the functions $g_{1/j,S(\varepsilon),M}$ are as in K\"orner's lemma. Now the properties $(i)-(iii)$ are immediate consequences of the corresponding properties of the $g$'s from K\"orner's lemma. Note that the statement in $(iv)$ allows follows from item $(iv)$ of K\"orner's Lemma, since our sum only involves terms $j>1/\varepsilon$. It therefore remains only to verify $(v)$. Since all functions involved are real-valued, we only need to estimate the positive Fourier coefficients. The Fourier terms will be decomposed into the following 3 terms: 

\[
\sum_{n>0} \abs{\widehat{\psi_\varepsilon}(n)}^2 \lambda_n = \left( \sum_{1\leq n < 1/\varepsilon} +\sum_{1/\varepsilon \leq n \leq S(\varepsilon)} + \sum_{n> S(\varepsilon)} \right)\abs{\widehat{\psi_\varepsilon}(n)}^2 \lambda_n.
\]
In order to estimate the first term, we shall utilize the estimate in $(iv)$ of K\"orner's lemma, in conjunction with $(c.)$ of \thref{LEM:DOUBLE}, as follows:
\begin{multline*}
\sum_{1\leq n < 1/\varepsilon} \abs{\widehat{\psi_\varepsilon}(n)}^2 \lambda_n \leq \frac{1}{L(\varepsilon)^2}\sum_{1\leq n < 1/\varepsilon} \frac{1}{\lambda_n} \left( \sum_{1/\varepsilon \leq j \leq S(\varepsilon)} \frac{1}{j} \abs{\widehat{g}_{1/j,S(\varepsilon),M}(n)} \right)^2 \\ 
\lesssim \frac{A(M,S)^2}{L(\varepsilon)^2} \sum_{1\leq n <1/\varepsilon} \lambda_n n^{2M} \left( \sum_{1/\varepsilon \leq j \leq S(\varepsilon)} \frac{1}{\lambda_j j^{M+1}} \right)^2 \\ \lesssim \frac{A(M,S)^2}{L(\varepsilon)^2} \frac{\varepsilon^{2M}}{\lambda^2_{[1/\varepsilon]}} \sum_{1\leq n < 1/\varepsilon} \lambda_n n^{2M} \lesssim \frac{A(M,S)^2}{L(\varepsilon)^2} \frac{1}{\varepsilon \lambda_{[1/\varepsilon]}}\lesssim \frac{A(M,S)^2}{L(\varepsilon)}.
\end{multline*}
In the last step, again utilized the condition $(ii)$ of $(\lambda_n)_n$, but in the following way:
\begin{equation}\label{EQ:sumL}
\frac{1}{\varepsilon \lambda_{[1/\varepsilon]}} \asymp \sum_{1/\varepsilon\leq j \leq 2/\varepsilon} \frac{1}{\lambda_j} \leq L(\varepsilon).
\end{equation}
The third term is estimated in similar way as the first, using $(iv)$ of K\"orner's lemma, but now in conjunction with $(b.)$ of \thref{LEM:DOUBLE}:
\begin{multline*}
\sum_{n>S} \abs{\widehat{\psi_\varepsilon}(n)}^2 \lambda_n \leq \frac{A^2(M,S)}{L^2(\varepsilon)} \cdot \left(\sum_{n>S} \frac{\lambda_n}{n^{2M}} \right) \cdot \left( \sum_{1/\varepsilon \leq j \leq S} \frac{1}{\lambda_j} j^{M-1} \right)^2 \\ 
\leq \frac{A^2(M,S)}{L^2(\varepsilon)} \left( \frac{\lambda_{S}}{S^M} \sum_{n>S} \frac{1}{n^M}  \right) \frac{S^{2M}}{\lambda^2_{S}} \lesssim \frac{A^2(M,S)}{L^2(\varepsilon)}\frac{S(\varepsilon)}{\lambda_{S(\varepsilon)}} 
\lesssim \frac{A^2(M,S)}{L(\varepsilon)}.
\end{multline*}
In the last step, again an estimate as \eqref{EQ:sumL} involving $S(\varepsilon)$. In order to estimate the first sum, we shall need to further decompose the Fourier coefficients of $\psi_{\varepsilon}$ into the following two terms:
\begin{multline*}
\widehat{\psi_\varepsilon}(n) =
\frac{1}{L(\varepsilon)}  \sum_{1/\varepsilon \leq j\leq n} \frac{1}{\lambda_j} \widehat{g}_{1/j, S(\varepsilon), M}(n) +  \frac{1}{L(\varepsilon)}  \sum_{n<j\leq S(\varepsilon)} \frac{1}{\lambda_j} \widehat{g}_{1/j, S(\varepsilon), M}(n) =: \widehat{\psi_1}(n) + \widehat{\psi_2}(n) ,
\end{multline*}
for $1/\varepsilon \leq n \leq S(\varepsilon)$. Applying $(iv)$ of K\"orner's lemma in conjunction with $(b.)$ of \thref{LEM:DOUBLE}, we get 
\begin{multline*}
\sum_{1/\varepsilon \leq n \leq S} \abs{\widehat{\psi_1}(n)}^2 \lambda_n \leq \frac{A(M,S)^2}{L(\varepsilon)^2}\sum_{1/\varepsilon \leq n \leq S} \frac{\lambda_n}{n^{2M}} \left( \sum_{1/\varepsilon \leq j \leq n} \frac{j^{M-1}}{\lambda_j } \right)^2  \\ \lesssim \frac{A(M,S)^2}{L(\varepsilon)^2}\sum_{1/\varepsilon \leq n \leq S} \frac{1}{\lambda_n} = \frac{A(M,S)^2}{L(\varepsilon)}.
\end{multline*}
Arguing similarly for the second term, instead using $(c.)$ of \thref{LEM:DOUBLE}, implies
\[
\sum_{1/\varepsilon \leq n \leq S} \abs{\widehat{\psi_2}(n)}^2 \lambda_n \leq \frac{A(M,S)^2}{L(\varepsilon)^2}\sum_{1/\varepsilon \leq n \leq S} \lambda_n n^{2M} \left( \sum_{j>n} \frac{1}{\lambda_j j^{M+1} } \right)^2  \lesssim \frac{A(M,S)^2}{L(\varepsilon)}.
\]
Combining, we arrive at $\sum_{n>0}\abs{\widehat{\psi_\varepsilon}(n)}^2 \lambda_n \lesssim A(M,S(\varepsilon))^2/L(\varepsilon) \leq A(M,S(\varepsilon))^2 \varepsilon$. The proof follows by a simple re-scaling argument.
\end{proof}

%This will likely require an adaption of Lemma 21 in \cite{korner2003topological}. To this end, we shall 

\subsection{A Baire category argument}

In order to carry out the Baire category argument, we need to set up an appropriate functional theoretical framework.

\begin{lemma} \thlabel{LEM:Lw} The subset $\mathscr{S}_{\lambda} \subset \ell^2(\lambda)$ consisting of elements $f\in L^\infty(\T,dm)$ with 
\[ 
0\leq f(\zeta) \leq 2, \qquad \text{dm-a.e} \qquad \zeta \in \T,
\]
forms a closed subset of $\ell^2(\lambda)$.
\end{lemma}
\begin{proof}
The proof is simple, hence we only sketch it. Since the set $\mathscr{S}_{w}$ is convex, it suffices to show that the set is weakly closed. To this end, we note that for any positive $\varphi \in C^\infty(\T)$, one has 
\[
0 \leq \int_{\T} f \varphi dm \leq 2 \int_{\T} \varphi dm.
\]
Now choosing $\varphi$ to be Poisson kernels wrt the unit-disc $\{|z|<1\}$ and using standard properties of their boundary behavior, one easily concludes the proof.
\end{proof}
Note that $\mathscr{S}_{\lambda}$ only defines a cone in $\ell^{2}(\lambda)$. We denote by $\mathcal{L}_{\lambda} \subset \mathscr{S}_{\lambda} \times \mathscr{C}$ the collection of ordered pairs $(f,E)$ with $f\in \mathscr{S}_{\lambda}$ and $E$ compact set, such that
\[
\supp{f} \subseteq E.
\]
We now record the following simple lemma, whose proof is immediate from \thref{LEM:Lw} and standard properties of support. 

\begin{lemma}\thlabel{LEM:Lwcomplete} The set $\mathcal{L}_{\lambda}$ of ordered pair $(f,E)$ equipped with the metric 
\[
d_{\lambda}\left( (f,E), (g,K) \right) := \norm{f-g}_{\ell^2(\lambda)} + d_{\mathscr{C}}(E,K),
\]
becomes a complete metric space.
\end{lemma}

We outline the principal result in this subsection.

\begin{prop}\thlabel{PROP:MAINPROPl2w}  Let $(\lambda_n)_n$ positive numbers satisfying the hypothesis $(i)-(ii)$ from \thref{THM:BADSUPP}. For any $a\in \T$, consider the set
\[
\mathscr{E}_a = \left\{ (f,E) \in \mathcal{L}_\lambda: E \, \text{does not meet an open arc containing} \, a \right\}.
\]
Then $\mathscr{E}_a$ is an open and dense subset in the metric space $(\mathcal{L}_\lambda, d_\lambda)$.
\end{prop}
\begin{proof}
To see why $\mathscr{E}_a$ open, let $(f,E) \in \mathscr{E}_a$ and pick $\delta>0$ such that the arc $I_{2\delta}(a)$ centered at $a$ of length $4\delta$ does not meet $E$. Now if $d_{\lambda}\left( (f,E), (g,K) \right) \leq \delta/2$, then we in particular have that $d_{\mathscr{C}}(E,K) \leq \delta /2$, hence we can infer that $I_{\delta/2}(a) \cap K = \emptyset$. This shows that $\mathscr{E}_a$ is indeed open. \\

In order to verify that $\mathscr{E}_a$ is dense, it suffices to show that for any $(f,E) \in \mathcal{L}_\lambda$ and any $\delta>0$, there exists $(g,K) \in \mathscr{E}_a$ such that 
\[
d_{\lambda} \left( (f,E), (g,K) \right) < \delta.
\]
By means of re-scaling $f$ and employing a simple argument involving convolution with a smooth approximate of the identity, we may actually assume that $f \in C^\infty(\T)$ with $0\leq f(\zeta) \leq 2-\delta$ for all $\zeta \in \T$. We now invoke \thref{LEM:KEYKÖR}. There exists an integer $N(\lambda)>0$, such that for any $N> 100 N(\lambda)$, the follow statement holds: for any $0<\varepsilon<1$, there exists $\psi_{\varepsilon} \in C^\infty(\T)$ such that 
\begin{enumerate}
    \item[(i)] $0\leq \psi_\varepsilon\leq 1+\varepsilon$ on $\T$,
    \item[(ii)] $\psi_{\varepsilon}=0$ in a neighborhood $J_\varepsilon(1)$ of $\zeta=1$, whose length tends to zero as $\varepsilon \to 0$,
    \item[(iii)] $\int_{\T} \psi_\varepsilon dm = 1$,
    \item[(iv)]$\sup_{n\neq 0} \abs{\widehat{\psi_\varepsilon} (n)} \leq \min \left( \varepsilon, c(N) \varepsilon^{1-N} \abs{n}^{-N} \right)$,
    \item[(v)]
    $\sum_{n\neq 0} \abs{\widehat{\psi}_\varepsilon (n)}^2 \lambda_{|n|} \leq \varepsilon^2$.
\end{enumerate}
Now set $f_\varepsilon (\zeta) = f(\zeta) \cdot \psi_\varepsilon(\zeta \conj{a})$ and $E_\varepsilon = E \setminus J_{\varepsilon}(a)$. It easily follows that the pair $(f_\varepsilon, E_\varepsilon)$ belongs to $\mathscr{L}_{\lambda}$ and we have $d_{\mathscr{C}}(E,E_\varepsilon)\to 0$ as $\varepsilon \to 0+$. Therefore, it only remains to show that 
\[
\lim_{\varepsilon \to 0+} \sum_{n} \abs{\widehat{f_\varepsilon}(n)-\widehat{f}(n)}^2 \lambda_{|n|} = 0.
\]
Since all functions involved are real-valued, we only need to estimate the sum for $n\geq 0$. First we write
\[
\widehat{f_\varepsilon}(n)-\widehat{f}(n) = \sum_{m\neq n} \widehat{f}(m) \widehat{\psi_\varepsilon}(n-m) = \left(\sum_{|m|\leq n/2} + \sum_{\substack{|m|>n/2 \\ m\neq n}} \right) \widehat{f}(m) \widehat{\psi_\varepsilon}(n-m) =: \Sigma_1 + \Sigma_2.
\]
In order to make the second sum small, we use property $(iv)$ and the smoothness of $f$. Hence for any desirable value of $A>0$, there exists $C(A)>0$, such that
\[
\abs{\sum_{\substack{|m|>n/2 \\ m\neq n}} \widehat{f}(m) \widehat{\psi_\varepsilon}(n-m)} \leq \varepsilon \sum_{|m|>n/2} \abs{\widehat{f}(m)} \leq   \frac{C(A) \varepsilon}{(1+n)^{A}}.
\]
This in conjunction with the property $(a.)$ of $(\lambda_n)_n$ in \thref{LEM:DOUBLE} gives
\begin{equation}\label{EQ:2sum}
\sum_{n\geq 0} \abs{\sum_{\substack{|m|>n/2 \\ m\neq n}} \widehat{f}(m) \widehat{\psi_\varepsilon}(n-m)}^2 \lambda_{n} \leq C(A)^2 \varepsilon^2 \sum_{n\geq 0} \frac{\lambda_n}{(1+n)^{2A}} \leq C'(A) \varepsilon^2.
\end{equation}
It remains only to estimate the Fourier coefficients of $\Sigma_1$, which we shall split into two further pars.
%\[
%\sum_{n\geq 0}\abs{ \sum_{|m|\leq n/2} \widehat{f}(m) \widehat{\psi_\varepsilon}(n-m)}^2 \lambda_{n}.
%\]
%To this end, we shall split the above sum into two parts. 
Let $M(\varepsilon) \to +\infty$ as $\varepsilon \to 0+$ be a parameter to be specified later. For $0\leq n < M(\varepsilon)$, we make use of the property $(v)$ of $\psi_\varepsilon$, which gives
\begin{multline}\label{EQ:Meps1}
\sum_{0\leq n < M(\varepsilon)} \abs{\sum_{|m|\leq n/2} \widehat{f}(m) \widehat{\psi_\varepsilon}(n-m)}^2 \lambda_{n} \\ \leq 
\sum_{0\leq n < M(\varepsilon)} \sum_{|m|\leq n/2} \abs{\widehat{f}(m)}^2 \sum_{|m|\leq n/2} \abs{ \widehat{\psi_\varepsilon}(n-m)}^2 \lambda_{n} 
\\ \asymp \sum_{0\leq n < M(\varepsilon)} \sum_{|m|\leq n/2} \abs{\widehat{f}(m)}^2 \sum_{|m|\leq n/2} \abs{ \widehat{\psi_\varepsilon}(n-m)}^2 \lambda_{|n-m|}
\\ \leq \varepsilon^2 \sum_{0\leq n < M(\varepsilon)} \sum_{|m|\leq n/2} \abs{\widehat{f}(m)}^2  \leq
\varepsilon^2 M(\varepsilon) \norm{f}^2_{L^2}.
\end{multline}
In the first step we used Cauchy-Schwartz inequality for the inner sum, while in the second step we used the assumption $\lambda_{n} \asymp \lambda_{|n-m|}$ whenever $|m|\leq n/2$. Moving forward, it remains only to estimate the sum for $n\geq M(\varepsilon)$. This time, we instead make use of property $(iv)$, in the following way:
\begin{multline*}
\abs{ \sum_{|m|\leq n/2} \widehat{f}(m) \widehat{\psi_\varepsilon}(n-m)}^2 \leq C(N)^2 \varepsilon^{2-2N} \left( \sum_{|m|\leq n/2} \abs{\widehat{f}(m)}\abs{n-m}^{-N} \right)^2 \\
\lesssim C(N)^2 \varepsilon^{2-2N} n^{-2N} \norm{f}^2_{\ell^1}.
\end{multline*}
With this estimate at hand, we get 
\begin{equation}\label{EQ:Loptial}
\sum_{n\geq M(\varepsilon)} \abs{ \sum_{|m|\leq n/2} \widehat{f}(m) \widehat{\psi_\varepsilon}(n-m)}^2 \lambda_{n} \lesssim C(N)^2\norm{f}^2_{\ell^1} \varepsilon^{2-2N} \sum_{n\geq M(\varepsilon)} \frac{\lambda_n}{n^{2N}}.
\end{equation}
Utilizing part $(a.)$ of \thref{LEM:DOUBLE}, we obtain 
\[
\varepsilon^{2-2N} \sum_{n\geq M(\varepsilon)} \frac{\lambda_n}{n^{2N}} \lesssim \varepsilon^{2-2N} \frac{\lambda_{M(\varepsilon)}}{M(\varepsilon)^{N(\lambda)}} \sum_{n\geq M(\varepsilon)} \frac{1}{n^{2N-N(\lambda)}} \lesssim \varepsilon^{2-2N} \frac{\lambda_{M(\varepsilon)}}{M(\varepsilon)^{2N-1}}.
\]
Now choosing $M(\varepsilon) \asymp \varepsilon^{-3/2}$ and returning back to \eqref{EQ:Loptial}, we arrive at
\[
\sum_{n\geq M(\varepsilon)} \abs{ \sum_{|m|\leq n/2} \widehat{f}(m) \widehat{\psi_\varepsilon}(n-m)}^2 \lambda_{n} \lesssim \lambda_{\varepsilon^{-3/2} } \left( \varepsilon^{3/2} \right)^{(2N+1)/3} \to 0, \qquad \varepsilon\to 0+
\]
since $N> 100N(\lambda)$ in view of $(a.)$ of \thref{LEM:DOUBLE}. Now since \eqref{EQ:2sum} and \eqref{EQ:Meps1} can also be made arbitrary small, we conclude
\[
\lim_{\varepsilon \to 0+} \sum_{n\geq 0} \abs{\widehat{f_\varepsilon}(n)-\widehat{f}(n)}^2 \lambda_{|n|} =0.
\]
This proves that $\mathscr{E}_a$ is dense in the metric space $(\mathcal{L}_\lambda, d_\lambda)$.
\end{proof}

With this at hand, we easily complete the proof \thref{THM:BAIRETOP}.

\begin{proof}[Proof of \thref{THM:BAIRETOP}] Pick a countable dense subset $\{a_j\}_j \subset \T$ and note that according to \thref{PROP:MAINPROPl2w}, and the Baire category theorem, the set $\mathscr{E} = \cap_j \mathscr{E}_{a_j}$ is dense in $\mathcal{L}_{\lambda}$, hence non-empty. Now if $(f,E) \in \mathscr{E}$, then $E \cap \{a_j\}_j = \emptyset$ by construction, hence $E$ cannot have any interior point. Since $\supp{f} \subseteq E$, we conclude that $f$ satisfies the required properties.
\end{proof}  
\subsection{Functions with Sparse Fourier support} 

Here we briefly explain why techniques involving sparse Fourier support cannot simply prove \thref{THM:BADSUPP}. To this end, given a subset of integers $\Lambda$, we denote by $C_{\Lambda}(\T)$ the closed subspace of continuous functions $f$ on $\T$ with the property that $\supp{\widehat{f}}\subseteq \Lambda$. A set $\Lambda$ is said to be a \emph{Sidon set} if there exists $C(\Lambda)>0$, such that 
\[
\sum_{n} \abs{\widehat{f}(n)} \leq C(\Lambda) \sup_{\zeta \in \T} \abs{f(\zeta)}, \qquad \forall f\in C_{\Lambda}(\T).
\]
Examples of Sidon sets include lacunary sequences $\Lambda = \{N_k\}_k$ with $\inf_k N_{k+1}/N_k >1$, which is a classical result due to A. Zygmund. For instance, see Ch. 5 in \cite{zygmundtrigseries}, or Lemma 1.4 in Ch. V of \cite{katznelson2004introduction}. A remarkable arithmetic characterization of Sidon sets was given by G. Pisier in \cite{pisier1983arithmetic}. Given an integer $n$, we denote by $R(n,\Lambda)$ the numbers of ways to write $n= \sum_j \varepsilon_j \lambda_j $ as finite linear combinations of $\lambda_j \in \Lambda$ and with $\varepsilon_j \in \{-1,0,1\}$ for all $j$. G. Pisier proved that $\Lambda$ is a Sidon set if and only if there exists a number $0<\gamma<1$, such that for any finite subset $\Gamma \subset \Lambda$, we have 
\[
\sup_{n \in \mathbb{Z}} R(n, \Gamma) \leq 3^{\gamma \abs{\Gamma}},
\]
where $\abs{\Gamma}$ denotes the cardinality of the subset $\Gamma$. In other words, representations of integers involving finite sub-collections of $\Lambda$ must be exponentially sparser. This clarifies why methods involving sparse Fourier support, such as Riesz products, cannot readily be used to produce functions appearing in \thref{THM:BADSUPP}. For further connections to Riesz products, see also refer the reader to the work of J. Bourgain in \cite{bourgain1985sidon}. See also Kronecker's Theorem on independent sets in Ch. VI of \cite{katznelson2004introduction}.

\section{Indispensability of regularity hypothesis}\label{SEC:3}
Here, we give a short proof of \thref{THM:(ii)}. In fact, we shall derive it from the following stronger statement.
\begin{thm}\thlabel{THM:iistrong} Let $\Phi$ be a positive increasing function on $(0,\infty)$. There exists increasing positive numbers $(\lambda_n)_n$ with $\lambda_n \uparrow + \infty$, such that 
\[
\sum_n \Phi( 1/\lambda_n) = +\infty,
\]
but if $S$ is a non-zero distribution on $\T$ with $(\widehat{S}(n))_n \in \ell^2(\lambda)$, then $\supp{S}= \T$.
\end{thm}
The proof of \thref{THM:iistrong} rests on a lemma from the work of T. W. K\"orner, whose surprisingly short and simple proof is attributed to Y. Meyer.
\begin{lemma}[K\"orner-Meyer, Lemma 2.1 \cite{korner1977theorem2}] Given $N\geq 1$ and $\gamma, \delta >0$, there exists $\varepsilon=\varepsilon(N,\gamma,\delta)>0$, such that the following statement holds: whenever $S$ is a distribution on $\T$ with
\[
(i) \, \sum_{|n|\leq N} \abs{\widehat{S}(n)}^2 \geq \gamma, \qquad (ii) \, \sup_{|n|>N} \abs{\widehat{S}(n)} \leq \varepsilon,
\]
then $\sup_{\zeta \in \T} \dist{\zeta}{\supp{S}}\leq \delta$.
\end{lemma}
With this lemma at hand, we can give a short proof of our principal observation in this section, which is a modification of Theorem 1.2 in \cite{korner1977theorem2}, adapted to our setting.

\begin{proof}[Proof of \thref{THM:iistrong}] Set $N_1=\varepsilon_1=1$, and suppose that positive integers $N_1 < N_2 < \dots < N_k$, and positive numbers $\varepsilon_1>\varepsilon_2 > \dots > \varepsilon_k>0$ have already been constructed. Now pick $N_{k+1}>N_k$, such that 
\[
\left(N_{k+1}-N_k \right) \Phi(\varepsilon^2_{k} 2^{-2k-1} )\geq 1.
\]
Invoking the K\"orner-Meyer Lemma, we can find a number $0<\varepsilon_{k+1}< \varepsilon_k$ such that the following statement holds: whenever $S$ is a distribution on $\T$ with 
\[
(i)_k \, \sum_{|n|\leq N_{k+1}} \abs{\widehat{S}(n)}^2 \geq 2^{-k}, \qquad (ii)_k \, \sup_{|n|>N_{k+1}} \abs{\widehat{S}(n)} \leq \varepsilon_{k+1},
\]
then $\sup_{\zeta \in \T} \dist{\zeta}{\supp{S}} \leq 2^{-k}$. 
Define the corresponding sequence of positive numbers $(\lambda_n)_n$ as follows: For each $k=0,1,2,\dots$ set $\lambda_{N_k} = 2^{2k}/\varepsilon^2_k$, and interpolate the intermediate values of $1/\lambda_n$ linear-wise:
\[
\frac{1}{\lambda_{n}} := \frac{\varepsilon^2_k}{2^{2k}} \frac{N_{k+1}-n}{N_{k+1}-N_k}  + \frac{\varepsilon^2_{k+1}}{2^{2(k+1)}}   \frac{n-N_{k}}{N_{k+1}-N_k}, \qquad N_k \leq n < N_{k+1}, \qquad k=0,1,2, \dots.
\]
Clearly, $(\lambda_n)_n$ is increasing with $\lambda_n \uparrow + \infty$. Furthermore, it follows from the monotonicity of $\Phi$ in conjunction with the constructions of $(N_k)_k$ that 
\[
\sum_{N_k\leq n \leq (N_{k+1}+N_k)/2} \Phi(1/\lambda_n) \geq \frac{(N_{k+1}-N_k)}{2}\Phi(\varepsilon^2_k 2^{-2k-1} ) \geq \frac{1}{2}, \qquad k=0,1,2 \dots,
\]
hence $\sum_n \Phi(1/\lambda_n) = + \infty$. Now let $S$ be an arbitrary non-zero distribution on $\T$ with 
\[
\sum_n \abs{\widehat{S}(n)}^2 \lambda_{|n|} < \infty.
\]
We claim that both $(i)_k$ and $(ii)_k$ must hold for all sufficiently large $k$. Indeed, the former statement follows from $S\in \ell^2$, while the second property follows from the crude estimate:
\[
\sup_{|n|>N_{k+1}} \abs{\widehat{S}(n)}^2  \leq \frac{1}{\lambda_{N_{k+1}}} \sum_{|n|>N_{k+1}}\abs{\widehat{S}(n)}^2 \lambda_{|n|} \leq \frac{1}{\lambda_{N_{k+1}}} \leq \varepsilon^2_{k+1},
\]
which holds for sufficient large $k$. Therefore, $\sup_{\zeta \in \T} \dist{\zeta}{\supp{S}} \leq 2^{-k}$ holds for all sufficiently large $k$, and we conclude that $\supp{S}=\T$.

\end{proof}

\section{Appendix}
Here we outline simple proofs of \thref{THM:l2wcont} and \thref{COR:l2wlac}. %We shall only need the following lemma.
%\begin{lemma}[cf. page 125 in \cite{katznelson2004introduction}] \thlabel{LEM:Cla} For any integer $N_0>100$, there exists a real-valued trigonometric polynomial $T_0$, such that 
%\begin{enumerate}
    %\item[(i)] $\norm{T_0}_{L^\infty} \leq 1$,
    %\item[(ii)] $\norm{T_0}_{L^2}\geq 1/2$,
    %\item[(iii)] $\supp{\widehat{T}_0}\subseteq [0,N_0] \cap \mathbb{Z}$.
%\end{enumerate}
%\end{lemma}

%With this lemma at hand, we now turn to the proof.
\begin{proof}[Proof of \thref{THM:l2wcont}] Pick a real-valued trigonometric polynomial $T_0$ satisfying the following properties:
\begin{equation}
(i) \, \norm{T_0}_{L^\infty} \leq 1, \qquad (ii) \, \norm{T_0}_{L^2}\geq 1/2, 
\end{equation}
and let $N_0>0$ denote its degree. For instance, such polynomials can be found by taking $h(\zeta)= e^{i\Re(\zeta)}$ and passing to appropriate Fej\'er means. Now chose $n_0> 2N_0$ and take a subsequence $(n_j)_j$ with the properties that
\[ 
\sum_j \frac{1}{\sqrt{\lambda_{n_j}}} <\infty, \qquad n_{j+1} - n_j \geq 2N_0, \qquad j=0,1,2,\dots.
\]
Set $T_j(\zeta) = \zeta^{n_j}T_0(\zeta)$ for $j=0,1,2,\dots$, and consider a function $f$ of the form
\[
f(\zeta) = \sum_j \frac{1}{\sqrt{\lambda_{n_j}}} T_{j}(\zeta) , \qquad \zeta \in \T.
\]
Since $(1/\sqrt{\lambda_{n_j}})_j$ is summable, the Weierstrass M-test ensures that $f$ is continuous on $\T$. Using the mutually disjoint Fourier support of $(T_j)_j$, we also have
\[
\sum_n \abs{\widehat{f}(n)}^2 \lambda_{|n|} = \sum_j \frac{1}{\lambda_{n_j}} \sum_{n_j\leq |n| \leq n_j + N_0} \abs{\widehat{T_j}(n)}^2 \lambda_{|n|} \geq \sum_j \sum_n \abs{\widehat{T_j}(n)}^2 \geq \sum_j \frac{1}{4}=\infty,
\]
where we in the penultimate step also utilized the monotonicity of $\lambda_n$. 
\end{proof}

The proof of \thref{COR:l2wlac} can easily be deduced from \thref{THM:l2wcont} by means of duality. Indeed, if the statements is false, then its dual space $\ell^2(1/w)$ is contained in the space of continuous functions, contradicting \thref{THM:l2wcont}.

\bibliographystyle{siam}
\bibliography{mybib}

\begin{thebibliography}{1}

\bibitem{bourgain1985sidon}
{\sc J.~Bourgain}, {\em {Sidon sets and Riesz products}}, in Annales de
  l'institut Fourier, vol.~35, 1985, pp.~137--148.

\bibitem{evans2022partial}
{\sc L.~C. Evans}, {\em {Partial differential equations}}, vol.~19, American
  Mathematical Society, 2022.

\bibitem{havinbook}
{\sc V.~P. Havin and B.~J\"oricke}, {\em The uncertainty principle in harmonic
  analysis}, vol.~72 of Encyclopaedia Math. Sci., Springer, Berlin, 1995.

\bibitem{katznelson2004introduction}
{\sc Y.~Katznelson}, {\em {An introduction to harmonic analysis}}, Cambridge
  University Press, 2004.

\bibitem{korner1986theorem3}
{\sc T.~K{\"o}rner}, {\em {On The Theorem of Iva{\v{s}}ev-Musatov III}},
  Proceedings of the London Mathematical Society, 3 (1986), pp.~143--192.

\bibitem{korner2003topological}
\leavevmode\vrule height 2pt depth -1.6pt width 23pt, {\em {A topological
  Iva{\v{s}}ev-Musatov theorem}}, Journal of the London Mathematical Society,
  67 (2003), pp.~448--460.

\bibitem{korner1977theorem2}
{\sc T.~W. K\"orner}, {\em {On the theorem of {I}va\v sev-{M}usatov. {II}}},
  Ann. Inst. Fourier (Grenoble), 28 (1978), pp.~vi, 123--142.

\bibitem{pisier1983arithmetic}
{\sc G.~Pisier}, {\em {Arithmetic characterizations of Sidon sets}}, Bulletin
  of the American Mathematical Society, 8 (1983), pp.~87--89.

\bibitem{zygmundtrigseries}
{\sc A.~Zygmund}, {\em Trigonometric series. 2nd ed. {V}ols. I, II}, Cambridge
  University Press, New York, 1959.

\end{thebibliography}

\Addresses

\end{document}